\documentclass[12pt]{amsart}
\usepackage{amsmath, amsthm,amssymb}
\usepackage{latexsym}		% to get LASY symbols
\usepackage{graphicx}		% to insert PostScript figures
\usepackage{amssymb}
\usepackage{enumerate}
\usepackage{fancyhdr}

\theoremstyle{plain}
\newtheorem{theorem}{Theorem}[section]
\newtheorem{corollary}[theorem]{Corollary}

\newtheorem{proposition}[theorem]{Proposition}
\newtheorem{lemma}[theorem]{Lemma}

\theoremstyle{definition}
\newtheorem{definition}[theorem]{Definition}

\theoremstyle{remark}
\newtheorem{remark}[theorem]{Remark}

\begin{document}
\author{Sooran Kang}
\title[The Yang-Mills functional on quantum Heisenberg manifolds]
{The Yang-Mills functional and Laplace's equation on quantum Heisenberg manifolds}

\begin{abstract}
In this paper, we discuss the Yang-Mills functional and a certain family of its critical points on quantum Heisenberg manifolds using noncommutative geometrical methods developed by A. Connes and M. Rieffel.
In our main result, we construct a certain family of connections on a projective module over a quantum Heisenberg manifold that give rise to critical points of the Yang-Mills functional. Moreover, we show that this set of solutions can be described as a set of solutions to Laplace's equation on quantum Heisenberg manifolds.
\end{abstract}

\subjclass[2000]{Primary 46L87 ; Secondary 58B34}

\maketitle

Since Alain Connes initiated noncommutative differential geometry in his ground-breaking paper \cite{C1}, the theory has flourished in different areas and, has motivated new ideas in various fields. Connes and Marc Rieffel's Yang-Mills theory for the noncommutative torus \cite{CR} is one of these examples, using the framework of noncommutative geometry to extend Yang-Mills theory to finitely generated projective modules over non-commutative $C^{\ast}$-algebras. This generalization seems to be natural, but Connes' and Rieffel's Yang-Mills theory for the noncommutative torus seems to be the only specific example of such an application so far.
In this paper, using the same framework developed in \cite{CR}, we attempt to develop Yang-Mills theory on quantum Heisenberg manifolds, $\{D^{c,\hslash}_{\mu\nu}\}_{\hslash\in\mathbb{R}}$, which are a different type of noncommutative $C^{\ast}$-algebra first constructed by Marc Rieffel \cite{Rief1}. We are able to discover a certain family of critical points of the Yang-Mills functional on the quantum Heisenberg manifolds. Inspired by the works in \cite{KoSch} and \cite{Ros}, in particular by the work of J. Rosenberg, we also find that this set of critical points can be described as a set of solutions to Laplace's equation on quantum Heisenberg manifolds.

The main difference in our case from that of the theory of the noncommutative torus is the following. First of all, for fixed $\mu,\nu$ and $c$, the projective module over $D^{c,\hslash}_{\mu\nu}$ is constructed by realizing $D^{c,\hslash}_{\mu\nu}$ as a  generalized fixed point algebra of a certain crossed product $C^{\ast}$-algebra, and thereby developing a bimodule structure.
Also we use a particular Grassmannian connection to produce a compatible connection on the projective module. The method of finding such a non-trivial connection is related to the technique of finding Rieffel projections in noncommutative tori, a method not employed by Connes and Rieffel in \cite{CR} and \cite{Rief2}. In our case, the last step of finding actual solutions of the Yang-Mills equation is related to solving an elliptic partial differential equation, which is very different from the approach of \cite{CR} and \cite{Rief2}.

This paper is organized as follows. In Section 1, we begin with the definition of quantum Heisenberg manifolds, and we give a specific formula for a particular $E^{c,\hslash}_{\mu\nu}$-$D^{c,\hslash}_{\mu\nu}$ projective bimodule $\Xi$ described in \cite{Ab1} and \cite{Ab2}. In Section 2, we show that $E^{c,\hslash}_{\mu\nu}$ is isomorphic in a fashion preserving the bimodule structure to $D^{c,\hslash}_{\frac{1}{4\mu},\frac{\nu}{2\mu}}$, using the technique of crossed products by Hilbert $C^{\ast}$-bimodules described in \cite{AEE}, \cite{Ab4}. In Section 3, we describe the noncommutative geometrical framework for Yang-Mills theory, and we produce a special function $R$ that gives a non-trivial Grassmannian connection and curvature. In Section 4 and 5, we introduce the notion of ``multiplication-type'' element, and we describe a certain set of critical points of the Yang-Mills functional on quantum Heisenberg manifolds. In the last section, we show that the set of critical points that we found can be described as a set of solutions to Laplace's equation on quantum Heisenberg manifolds.\\

\textbf{Acknowledgement.} I would like to take this opportunity to thank my thesis advisor, Judith Packer, for her constant patience and encouragement, as well as a number of helpful suggestions and comments.

\section{Projective modules over quantum Heisenberg manifolds}

Let $G$ be the Heisenberg group, parametrized by
\[(x,y,z)= \begin{pmatrix}1&y&z\\0&1&x\\0&0&1\end{pmatrix}\]
so that when we identify $G$ with $\mathbb{R}^3$ the product is given by
\begin{equation*}
(x,y,z)(x',y',z')=(x+x',y+y',z+z'+yx').
\end{equation*}

For any positive integer $c$, let $D_c$ denote the subgroup of $G$ consisting of those $(x,y,z)$ such that $x$, $y$, and $cz$ are integers. Then the Heisenberg manifold, $M_c$, is the quotient $G/D_c$, on which $G$ acts on the left.

 In \cite{Rief1}, Rieffel constructed strict deformation quantizations $\{D^{c,\hslash}_{\mu\nu}\}_{\hslash \in \mathbb{R}}$ of $M_c$ in the direction of the Poisson bracket $\Lambda_{\mu\nu}$, determined by two real parameters $\mu$ and $\nu$, where ${\mu}^2+{\nu}^2\ne 0$. He recognized that these non-commutative $C^{\ast}$-algebras could be described as generalized
fixed-point algebras of certain crossed product $C^{\ast}$-algebras under proper actions. Then Abadie showed in \cite{Ab1} that it is possible to construct a (finitely generated) projective bimodule over two generalized fixed-point algebras, under appropriate conditions. As an example in \cite{Ab1} and \cite{Ab2}, she stated explicit formulas for a finitely generated projective module over two generalized fixed point algebras, one of which is $D^{c,\hslash}_{\mu\nu}$, the quantum Heisenberg manifold. We give here more details of the specific construction, which will be used for further discussion in later sections.

First, we introduce the reparametrization of Heisenberg group described in \cite{Rief1}. For a given positive integer $c$, we reparametrize the Heisenberg group $G$ as
\begin{equation}\label{rep-H}
(x,y,z)= \begin{pmatrix}1&y&z/c\\0&1&x\\0&0&1\end{pmatrix}.
\end{equation}
Then the product on $\mathbb{R}^3$ becomes
\[(x,y,z)(x',y',z')=(x+x',y+y',z+z'+cyx'),\]
and $D_c$ becomes the subgroup with integer entries. Let $E_c=\{(0,m,n)\in D_c\}$ be the normal subgroup of $D_c$. Then we can check that for $f \in C^{\infty}(G)$, the operator corresponding to right translation of $f$ by $(k,m,n)\in D_c$ is given by
\[f(x,y,z)\longrightarrow f(x+k,y+m,z+n+cky).\]
To obtain the Heisenberg manifold, we consider the quotient, $N_c$, of $G$ by the right action of $E_c$. Then this quotient looks like $\mathbb{R}\times \mathbb{T}^2$. If we define an action $\rho$ of ${\mathbb Z}$ on $N_c$ by
\begin{equation}\label{proper1}
(\rho_k f)(x,y,z)=f(x+k,y,z+cky),
\end{equation}
for $(x,y,z)\in \mathbb{R}\times \mathbb{T}^2$ and a smooth function $f$ on $N_c$, then the Heisenberg manifold $M_c$ is the quotient of $N_c$ by $\rho$. Also the action $\rho$ of ${\mathbb Z}$ can be viewed as $(k,0,0)\in D_c$ acting on the right on $N_c$, which means the following.
\[f((x,y,z)\cdot(k,0,0))=f(x+k,y,z+cky)=(\rho_k f)(x,y,z).\]
Thus we can consider functions on $M_c$ as functions on $N_c$ which are invariant under the action $\rho$. Now we describe the action of $G$ on the left on $N_c$. For $g=(r,s,t)\in G$ and $(x,y,z)\in N_c = \mathbb{R}\times \mathbb{T}^2$, define the left action of $G$ on $C^{\infty}(N_c)$ by
\begin{equation}\label{H-action1}
(g\cdot f)(x,y,z)=f((r,s,t)^{-1}\cdot(x,y,z))=f(x-r,y-s,z-t-sc(x-r)),
\end{equation}
where $f \in C^{\infty}(N_c)$ and $(r,s,t)^{-1}$ is the inverse of $(r,s,t)$ in $G$.
Then a straightforward calculation shows that this action of $G$ on the left on $N_c$ commutes with the action $\rho$.

To obtain a strict deformation quantization of the Heisenberg manifold, Rieffel first formed a deformation quantization on $N_c\cong \mathbb{R}\times \mathbb{T}$ in \cite{Rief1}, denoted by $A_\hbar$, by using the Fourier transform, and showed that the action $\rho$ on this quantization is proper. Then he recognized that $A_\hbar$ can be identified with a certain crossed product $C^{\ast}$-algebra under the map $J$ given in \cite{Rief1},
p. 547, and a strict deformation quantization of $C^{\infty}(M_c)$, denoted by $D_\hslash$, via the above isomorphism, it is possible to view as the generalized fixed-point algebra of this crossed product $C^{\ast}$-algebra under the action $\rho$. See more about proper actions and generalized fixed point algebras in \cite{Rief4}. The corresponding action $\rho$ and the action of the Heisenberg group on $A_\hslash$ are given as follows. Take the Fourier transform in the third variable in equations (\ref{proper1}) and (\ref{H-action1}); then we have, for $\phi \in S(\mathbb{R}\times \mathbb{T}\times \mathbb{Z}$),
\begin{equation}\label{proper2} (\rho_{k}\phi)(x,y,p)=\overline{e}(ckpy)\phi(x+k,y,p),\end{equation}
 and the formula for the action of the Heisenberg group on the same deformed algebra is given by
\begin{equation}\label{H-action2}L_{(r,s,t)}\phi(x,y,p)=e(p(t+cs(x-r))\phi(x-r,y-s,p),\end{equation}
where $S(\mathbb{R}\times \mathbb{T}\times \mathbb{Z})$ is Schwartz space, the set of functions on $\mathbb{R}\times \mathbb{T}\times \mathbb{Z}$ which go to zero at infinity faster than any polynomial grows.

  Since these two actions commute on $S(\mathbb{R}\times \mathbb{T}\times \mathbb{Z})$, a dense subalgebra of $A_{\hslash}$, the same formula $L$ gives the action of the Heisenberg group on the generalized fixed point algebra for $\rho$, $D_\hbar$.

 Now we state the specific formula for a particular projective module over the quantum Heisenberg manifolds, $D^{c,\hslash}_{\mu\nu}$ shown in
\cite{Ab1} and \cite{Ab2} as follows.

 Let $M=\mathbb{R}\times \mathbb{T}$ and $\lambda$ and $\sigma$ be the commuting actions of $\mathbb{Z}$ on $M$ defined by
\begin{equation*}
\lambda_p(x,y)=(x+2\hslash p\mu,y+2\hslash p\nu)\quad\text{and}\quad
\sigma_p(x,y)=(x-p,y),
\end{equation*}
where $\hslash$ is Planck's constant, $\mu, \nu \in \mathbb{R}$, and $p\in \mathbb{Z}$.

Then construct the crossed product $C^{\ast}$-algebras $C_b(\mathbb{R}\times \mathbb{T})\times_{\lambda}\mathbb{Z}$ and $C_b(\mathbb{R}\times \mathbb{T})\times_{\sigma}\mathbb{Z}$ with usual star-product and involution. Here $C_b(\mathbb{R}\times\mathbb{T})$ is a set of bounded functions on $\mathbb{R}\times\mathbb{T}$, and  $\rho$ and $\gamma$ denote the actions of $\mathbb{Z}$ on $C_{b}(\mathbb{R}\times
\mathbb{T})\times_{\lambda}\mathbb{Z}$ and $C_{b}(\mathbb{R}\times \mathbb{T})\times_{\sigma}\mathbb{Z}$ given by, for $\Phi,\Psi \in C_c(\mathbb{R}\times \mathbb{T}\times \mathbb{Z})$,
\begin{equation*}
(\rho_k\Phi)(x,y,p)=\overline{e}(ckp(y-\hslash p\nu))\Phi(x+k,y,p),
\end{equation*}
\begin{equation*}
(\gamma_k\Psi)(x,y,p)=e(cpk(y-\hslash k\nu))\Psi(x-2\hslash k\mu,y-2\hslash k\nu),
\end{equation*}
 where $k, p \in \mathbb{Z}$, and $e(x)=exp(2\pi ix)$ for any real number $x$.
Then these actions $\rho$, $\gamma$ are proper.
 The generalized fixed point algebra of $C_{b}(\mathbb{R}\times\mathbb{T})\times_{\lambda}\mathbb{Z}$ by the action $\rho$, denoted by $D^{c,\hslash}_{\mu\nu}$, is the closure of $\ast$-subalgebra $D_0$ in the multiplier algebra of $C_{b}(\mathbb{R}\times
\mathbb{T})\times_{\lambda}\mathbb{Z}$  consisting of functions $\Phi \in C_c(\mathbb{R}\times \mathbb{T}\times \mathbb{Z})$, which have compact support on $\mathbb{Z}$ and satisfy $\rho_k(\Phi)=\Phi$ for all $k \in \mathbb{Z}$.

\begin{remark}
The above formula of $\rho$ on $C_{b}(\mathbb{R}\times\mathbb{T})\times_{\lambda}\mathbb{Z}$ can be obtained from the equation (\ref{proper2}) under the map $J$ given in \cite{Rief1}, p. 547, and we consider $D^{c,\hslash}_{\mu\nu}$ as the corresponding generalized fixed point algebra, $D_{\hslash}$ under the same map $J$.
\end{remark}
We can obtain the action of the Heisenberg group on $D^{c,\hslash}_{\mu\nu}$ from the equation (\ref{H-action2}) via the map $J$, given by
\begin{equation}\label{H-action3}
(L_{(r,s,t)}\Phi)(x,y,p)=e(p(t+cs(x-r-{\hslash}p\mu)))\Phi(x-r,y-s,p),
\end{equation}
for $\Phi \in D_0$.

 Similarly, the generalized fixed point algebra of $C_{b}(\mathbb{R}\times \mathbb{T})\times_{\sigma}\mathbb{Z}$ by the action $\gamma$, denoted by $E^{c,\hslash}_{\mu\nu}$, is the closure of $\ast$-subalgebra $E_0$ in the multiplier algebra of $C_{b}(\mathbb{R}\times \mathbb{T})\times_{\sigma}\mathbb{Z}$  consisting of functions $\Psi \in C_c(\mathbb{R}\times \mathbb{T}\times \mathbb{Z})$, with compact support on $\mathbb{Z}$ and satisfying $\gamma_k(\Psi)=\Psi$ for all $k \in \mathbb{Z}$.

According to the main theorem in \cite{Ab1}, these generalized fixed point algebras $D^{c,\hslash}_{\mu\nu}$ and $E^{c,\hslash}_{\mu\nu}$ are strongly Morita equivalent. Let $\Xi$ be the left-$E^{c,\hslash}_{\mu\nu}$ and right-$D^{c,\hslash}_{\mu\nu}$ bimodule constructed as follows. $\Xi$ is the completion of $C_c(\mathbb{R}\times \mathbb{T})$ with respect to either one of the norms induced by one of the $D^{c,\hslash}_{\mu\nu}$ and $E^{c,\hslash}_{\mu\nu}$-valued inner products, $\langle\cdot,\cdot\rangle_D$ and $\langle\cdot,\cdot\rangle_E$ respectively, given by
\begin{equation*}\label{D-value-inner}
\langle f,g \rangle_D(x,y,p)=\sum_{k \in \mathbb{Z}}\overline{e}(ckp(y-\hslash p\nu))f(x+k,y)\overline{g}(x-2\hslash p\mu+k,y-2\hslash p\nu)\,,
\end{equation*}
\begin{equation*}\label{E-value-inner}
\langle f,g \rangle_E(x,y,p)=\sum_{k \in \mathbb{Z}}e(cpk(y-\hslash k\nu)\overline{f}(x-2\hslash k\mu,y-2\hslash k\nu)g(x-2\hslash k\mu+p,
y-2\hslash k\nu),
\end{equation*}
 where $f,g \in C_c(\mathbb{R}\times \mathbb{T})$ and $k,\,p \in \mathbb{Z}$.
Also the left and right action of $E^{c,\hslash}_{\mu\nu}$ and $D^{c,\hslash}_{\mu\nu}$ on $\Xi$ are given by
\begin{equation*}\label{left-action}
(\Psi\cdot f)(x,y)=\sum_{q \in\mathbb{Z}}\overline{\Psi}(x,y,q)f(x+q,y),
\end{equation*}
\begin{equation*}\label{right-action}
(g\cdot\Phi)(x,y)=\sum_{q\in\mathbb{Z}}g(x+2\hslash q\mu,y+2\hslash q\nu)\overline{\Phi}(x+2\hslash q\mu,y+2\hslash q\nu,q),
\end{equation*}
for $\Psi \in E_0$ , $\Phi \in D_0$ and $f, g \in \Xi$.

\section{Morita equivalence of quantum Heisenberg manifolds}
 It has been shown that the generalized fixed-point algebra of a certain crossed $C^{\ast}$-algebra constructed by Rieffel in \cite{Rief4}  can be generated by the fixed-point algebra and the first spectral subspace for the action of $\mathbb{T}$ on the crossed product $C^{\ast}$-algebra. Also, it is known that the first spectral subspace has a natural bimodule structure over the fixed-point algebra. Thus, we can classify generalized fixed-point algebras by examining each fixed-point algebra, its first spectral subspace, and the bimodule structure. We follow the same technique that Abadie introduced in the papers \cite{Ab4} and \cite{AEE} in order to prove that $E^{c,\hslash}_{\mu\nu}$ can be identified with a quantum Heisenberg manifold with parameters $\bigl(\frac{1}{4\mu},\frac{\nu}{2\mu}\bigl)$. Note that Abadie did not indicate that the generalized fixed-point algebra $E^{c,\hslash}_{\mu\nu}$ can be identified with  $D^{c,\hslash}_{\frac{1}{4\mu}\frac{\nu}{2\mu}}$  in the papers \cite{Ab4}, \cite{AEE}.

As we described earlier, the quantum Heisenberg manifold $D^{c,\hslash}_{\mu\nu}$ is the generalized fixed-point algebra of $C_b(\mathbb{R}\times\mathbb{T})\times_{\lambda}\mathbb{Z}$ under the action $\rho$. In particular,
\[D^{c,\hslash}_{\mu\nu}=\overline{span}\{\Phi\in C_c(\mathbb{R}\times\mathbb{T}\times\mathbb{Z})\;|\;\overline{e}(ckp(y-\hslash p\nu))\Phi(x+k,y,p)=\Phi(x,y,p)\;\;\text{for all}\;\;k\in\mathbb{Z}\}\]
\[=\overline{span}\{\phi\delta_p\;|\;\overline{e}(ckp(y-\hslash p\nu))\phi(x+k,y)=\phi(x,y)\;\;\text{for all}\;\; k\in\mathbb{Z}\}.\]
Now choosing $k=1$, then we can write $D^{c,\hslash}_{\mu\nu}$ by
\[D^{c,\hslash}_{\mu\nu}=\overline{span}\{\phi\delta_p\;|\;\overline{e}(cp(y-\hslash p\nu))\phi(x+1,y)=\phi(x,y)\},\]
for $\phi\in C_b(\mathbb{R}\times\mathbb{T})$.
Also notice that $D^{c,\hslash}_{\mu\nu}$ carries a natural dual action of $\mathbb{T}$.
Let $\varsigma$ be the action of $\mathbb{T}$ on $D^{c,\hslash}_{\mu\nu}$ given by
\[(\varsigma_z\Phi)(x,y,p)=z^p\Phi(x,y,p)=e(pz)\Phi(x,y,p)\;\;\text{for}\;\; z\in\mathbb{T}.\]
 Thus, the $n$th spectral subspace of $D^{c,\hslash}_{\mu\nu}$ is given as follows.
 \begin{equation}\label{nth}
(D^{c,\hslash}_{\mu\nu})_n=\{f\delta_n\;|\;f\in C_b(\mathbb{R}\times\mathbb{T}),\;\overline{e}(cn(y-\hslash n\nu)) f(x+1,y)=f(x,y)\}.
\end{equation}
Let $(D^{c,\hslash}_{\mu\nu})_0$ be the fixed point algebra of $D^{c,\hslash}_{\mu\nu}$ under the action $\varsigma$, and let $(D^{c,\hslash}_{\mu\nu})_1$ be the first spectral subspace of $D^{c,\hslash}_{\mu\nu}$ for $\varsigma$. i.e.
\[(D^{c,\hslash}_{\mu\nu})_0=\{\phi\delta_0\in D^{c,\hslash}_{\mu\nu}\;|\; \phi(x,y)=\phi(x+1,y)\},\]
\[(D^{c,\hslash}_{\mu\nu})_1=\{g\delta_1\in D^{c,\hslash}_{\mu\nu}\;|\; g(x,y)=\overline{e}(c(y-\hslash\nu))g(x+1,y)\}.\]
Similarly,  $E^{c,\hslash}_{\mu\nu}$, mentioned in the previous section, is the generalized fixed-point algebra of $C_b(\mathbb{R}\times\mathbb{T})\times_{\sigma}\mathbb{Z}$ under $\gamma$. Thus, it carries a natural action $\varrho$ of $\mathbb{T}$ given by
 \[(\varrho_z\Psi)(x,y,p)=z^p\Psi(x,y,p)\;\;\text{for}\;\; z\in\mathbb{T}.\]
 Let $(E^{c,\hslash}_{\mu\nu})_0$ be the fixed point algebra of $E^{c,\hslash}_{\mu\nu}$ under the action $\varrho$, and let $(E^{c,\hslash}_{\mu\nu})_1$ be the first spectral subspace of $E^{c,\hslash}_{\mu\nu}$ for $\varrho$. i.e.
\[(E^{c,\hslash}_{\mu\nu})_0=\{\psi\delta_0\in E^{c,\hslash}_{\mu\nu}\;|\;\psi(x,y)=\psi(x-2\hslash\mu, y-2\hslash\nu)\},\]
\[(E^{c,\hslash}_{\mu\nu})_1=\{f\delta_1\in E^{c,\hslash}_{\mu\nu}\;|\;f(x,y)=e(c(y-\hslash\nu))f(x-2\hslash\mu,y-2\hslash\nu)\}.\]
According to Proposition 1.2 in \cite{Ab4} by Abadie, the $C^{\ast}$-algebras $C_b(M/\alpha)\rtimes X^{\alpha,u}_\beta$ and  $C_b(M/\beta)\rtimes X^{\beta,u^{\ast}}_\alpha$ are Morita equivalent, where $\alpha$, $\beta$ are free and proper commuting actions of $\mathbb{Z}$ on a locally compact Hausdorff space $M$, $u$ is a unitary in $C_b(M)$, and  $X^{\alpha,u}_\beta$ and $X^{\beta,u^{\ast}}_\alpha$ are $C^{\ast}$-bimodules. Thus, there is a left-$C_b(M/\alpha)\rtimes X^{\alpha,u}_\beta$ and right-$C_b(M/\beta)\rtimes X^{\beta,u^{\ast}}_\alpha$ bimodule, and we denote it by ${\Xi}'$.
By choosing appropriate actions $\alpha$ and $\beta$, Abadie showed that $D^{c,\hslash}_{\mu\nu}$ is strongly Morita equivalent to $D^{c,\hslash}_{\frac{1}{4\mu},\frac{\nu}{2\mu}}$ in \cite{Ab4}. We give the specific formulas that we use later in this section as follows.

Consider the actions $\alpha$ and $\beta$ on $M=\mathbb{R}\times\mathbb{T}$ given by
$\alpha(x,y)=(x+\frac{1}{2\hslash\mu},y)$ and $\beta(x,y)=(x+1,y+2\hslash\nu)$.
Let $u(x,y)=\overline{e}(c(y-\hslash\nu))$, so $u^{\ast}(x,y)=e(c(y-\hslash\nu))$. Then $C_b(M/\alpha)\cong C(\mathbb{T}^2)$ and $C_b(M/\beta)\cong C(\mathbb{T}^2)$. Note that the unitary $u$ here is different from the unitary given in Proposition 2.2 in \cite{Ab4}.
With the unitary $u$ and the commuting actions $\alpha$ and $\beta$ as above, we can write the corresponding $C^{\ast}$-bimodules $X^{\alpha,u}_\beta$ and $X^{\beta,u^{\ast}}_\alpha$ as follows.
\[X^{\alpha,u}_\beta=\{f\in C_b(\mathbb{R}\times\mathbb{T})|f(x-\frac{1}{2\hslash\mu},y)=\overline{e}(c(y-\hslash\nu))f(x,y)\}.\]
Then $X^{\alpha,u}_\beta$ is a bimodule over $C_b(\mathbb{R}\times\mathbb{T}/\alpha)\cong C(\mathbb{T}^2)$ with following formulas.
For $\psi \in C_b(\mathbb{R}\times\mathbb{T}/\alpha)$ and  $f,g\in X^{\alpha,u}_\beta$,
\[(\psi\cdot f)(x,y)=\psi(x,y)f(x,y),\;\;(f\cdot\psi)(x,y)=f(x,y)\psi(x-1,y-2\hslash\nu),\]
\[\langle f,g\rangle_L(x,y)=f(x,y)\overline{g}(x,y),\;\;\langle f,g\rangle_R(x,y)=\overline{f}(x+1,y+2\hslash\nu)g(x+1,y+2\hslash\nu).\]
Here the subscripts $R$ and $L$ stand for the right and left inner products, respectively.
Also we can write $X^{\beta,u^{\ast}}_\alpha$ in the following way.
\[X^{\beta,u^{\ast}}_\alpha=\{g\in C_b(\mathbb{R}\times\mathbb{T}/\beta)|g(x-1,y-2\hslash\nu)=e(c(y-\hslash\nu))g(x,y)\}.\]
Then $X^{\beta,u^{\ast}}_\alpha$ is a bimodule over $C_b(\mathbb{R}\times\mathbb{T}/\beta)\cong C(\mathbb{T}^2)$ with  the following formulas.
For $\phi \in C_b(\mathbb{R}\times\mathbb{T}/\beta)$ and $f,g\in X^{\beta,u^{\ast}}_\alpha$,
\[(\phi\cdot f)(x,y)=\phi(x,y)f(x,y),\;\;(f\cdot\phi)(x,y)=f(x,y)\phi(x-\frac{1}{2\hslash\mu},y),\]
\[\langle f,g\rangle_L(x,y)=f(x,y)\overline{g}(x,y),\;\;\langle f,g\rangle_R(x,y)=\overline{f}(x+\frac{1}{2\hslash\mu},y)g(x,y).\]
For the $D^{c,\hslash}_{\mu\nu}$-$E^{c,\hslash}_{\mu\nu}$ bimodule $\Xi$ given in the previous section, and the $C_b(M/\alpha)\rtimes X^{\alpha,u}_\beta$-$C_b(M/\beta)\rtimes X^{\beta,u^{\ast}}_\alpha$ bimodule ${\Xi}'$ above,
it is not hard to verify that $D^{c,\hslash}_{\mu\nu}$ is strongly Morita equivalent to $D^{c,\hslash}_{\frac{1}{4\mu},\frac{\nu}{2\mu}}$ by showing that $D^{c,\hslash}_{\mu\nu}$ is isomorphic to $C_b(\mathbb{R}\times\mathbb{T}/\beta)\rtimes X^{\alpha,u}_{\beta}$ and $D^{c,\hslash}_{\frac{1}{4\mu},\frac{\nu}{2\mu}}$ is isomorphic to $C_b(\mathbb{R}\times\mathbb{T}/\beta)\rtimes X^{\beta, u^{\ast}}_{\alpha}$ in a fashion preserving the bimodule structures as shown in \cite{Ab4}. Thus we only need the following lemma to show that $E^{c,\hslash}_{\mu\nu}$ is isomorphic to $D^{c,\hslash}_{\frac{1}{4\mu},\frac{\nu}{2\mu}}$.
\begin{lemma}\label{MoritaE}
$E^{c,\hslash}_{\mu\nu}$ is isomorphic in a fashion preserving the bimodule structure to $C_b(\mathbb{R}\times\mathbb{T}/\beta)\rtimes X^{\beta, u^{\ast}}_{\alpha}$.
\end{lemma}
\begin{proof}
In this proof, we absorb the Planck constant $\hslash$ into the parameters $\mu$ and $\nu$ for simplicity.
 Define maps $S$ and $H$ by
\[S:X^{\beta, u^{\ast}}_{\alpha}\longrightarrow (E^{c,\hslash}_{\mu\nu})_1,\;\; H :C_b(\mathbb{R}\times\mathbb{T}/\beta)\longrightarrow C(\mathbb{T}^2),\]
\[S(f)(x,y)=e(\frac{cy^2}{2\nu})f(-\frac{x}{2\mu},-y),\;\;H(\phi)(x,y)=\phi(-\frac{x}{2\mu},-y),\]
for $f \in X^{\beta, u^{\ast}}_{\alpha}$ and $\phi \in C_b(\mathbb{R}\times\mathbb{T}/\beta)$.
The straightforward calculations show that $S$ and $H$ are bijections and $S(f)\in (E^{c,\hslash}_{\mu\nu})_1$.
Also, it is not hard to show that $S$ and $H$ satisfy the following conditions,
\begin{equation*}
S(\phi\cdot f)=H(\phi)\cdot S(f),\;\;S(f\cdot\phi)=S(f)\cdot H(\phi),
\end{equation*}
\begin{equation*}
\langle S(f),S(g)\rangle_L=H(\langle f,g\rangle_L),\;\;\langle S(f),S(g)\rangle_R=H(\langle f,g\rangle_R).
\end{equation*}
 Therefore $E^{c,\hslash}_{\mu\nu}$ is isomorphic to $C_b(\mathbb{R}\times\mathbb{T}/\beta)\rtimes X^{\beta, u^{\ast}}_{\alpha}$ in a fashion preserving the bimodule structure.
\end{proof}

\begin{proposition}
Let $\Xi$  be $E^{c,\hslash}_{\mu\nu}$-$D^{c,\hslash}_{\mu\nu}$ bimodule given in the previous section.
Then $E^{c,\hslash}_{\mu\nu}$ is isomorphic in a fashion preserving the bimodule structure to $D^{c,\hslash}_{\frac{1}{4\mu},\frac{\nu}{2\mu}}$.
\end{proposition}
\begin{proof}
Since $D^{c,\hslash}_{\frac{1}{4\mu},\frac{\nu}{2\mu}}$ is isomorphic to $C_b(\mathbb{R}\times\mathbb{T}/\beta)\rtimes X^{\beta, u^{\ast}}_{\alpha}$, Lemma \ref{MoritaE} shows that $E^{c,\hslash}_{\mu\nu}$ is isomorphic to $D^{c,\hslash}_{\frac{1}{4\mu},\frac{\nu}{2\mu}}$.
\end{proof}

\section{Grassmannian connections and curvature on quantum Heisenberg Manifolds}

As mentioned in \cite{C2}, projective modules are the proper generalizations of vector bundles when taking the view point of noncommutative geometry. By the Gelfand-Naimark theorem we can identify a commutative $C^{\ast}$-algebra $A$ with $C(X)$ for a locally compact space $X$, and the Serre-Swan theorem says that there is a  1-1 correspondence between the category of (smooth) vector bundles over a compact manifold with bundle maps and the category of finitely generated projective modules over commutative algebras with module morphisms. So, many times, a geometric property on X can be understood in terms of a corresponding algebraic property on the algebra $C(X)$. J. L. Koszul established an algebraic version of differential geometry, in particular, algebraic concepts for connections and curvature on a projective module over a commutative associative algebra in \cite{Kosz}. Before we state Connes' definitions of connections and curvature on a projective module over a noncommutative $C^{\ast}$-algebra, we mention the general algebraic definition of the connection on a module described in \cite{Kosz}.

Let $\Bbbk$ be a commutative ring with unit and $A$ be a commutative, associative algebra over $\Bbbk$ having unit element. Then a derivation $X$ is an element of $Hom_{\Bbbk}(A,A)$ with the condition, $X(ab)=(Xa)b+a(Xb)$. Denote the set of derivations by $D$; then $D$ is obviously an $A$-module with usual operations and has a natural Lie algebra structure.
A connection on $A$-module(commutative case) is described as a derivation law in \cite{Kosz}.
\begin{definition}\cite{Kosz}
A \textit{derivation law} $\nabla$ is an element of $Hom_A(D,Hom_{\Bbbk}(M,M))$, where $M$ is a unitary $A$-module, satisfying
\begin{enumerate}
 \item ${\nabla}_{X+Y}={\nabla}_X+{\nabla}_Y$, ${\nabla}_{aX}=a{\nabla}_X$,
 \item   ${\nabla}_X(au)=(Xa)u + a{\nabla}_Xu$,
\end{enumerate}
 for $a \in A$ , $u \in M$ and $X,Y \in D$.
\end{definition}
As can be seen, $\nabla_X$ is defined on a $A$-module $M$. So $\nabla_X$ can be viewed as a differentiation of differentiable sections of a bundle in a certain direction, since the set of differentiable sections of a vector bundle has an obvious module structure.
 Thus, this derivation law corresponds to a connection on a projective module in modern geometry, and Connes' definition of a connection can be viewed as a noncommutative extension of Koszul's derivation law.

   Now we state the setting developed by Connes in \cite{C1}. Let $G$ be a Lie group with Lie algebra ${\mathfrak g}$, and let $\alpha$ be an action of $G$ as automorphisms of a $C^{\ast}$-algebra $A$. We let $A^{\infty}=\{a \in A\;|\;g\mapsto\alpha_{g}(a)\;\text{is smooth in norm}\}$. Then the infinitesimal form of $\alpha$ gives an action, $\delta$, of the Lie algebra ${\mathfrak g}$ of $G$, as derivations on $A^{\infty}$.

By Lemma 1 in \cite{C1}, every finitely generated projective $A$-module $\Xi$ has a smooth dense sub-projective module ${\Xi}^{\infty}$ over $A^{\infty}$ such that $\Xi$ is isomorphic to ${\Xi}^{\infty}\otimes_{A^{\infty}}A$. Moreover, $A^{\infty}$ is dense in $A$ and ${\Xi}^{\infty}$ is dense in $\Xi$. Note that we say ``projective'' when we mean ``finitely generated projective''. Also we will denote ${\Xi}^{\infty}$ and $A^{\infty}$ by  $\Xi$ and $A$ for notational simplicity for the general definitions.

 Also, we can always equip $\Xi$ with an $A$-valued positive definite inner product $\langle \cdot,\cdot  \rangle_A $, called a Hermitian metric, such that $\langle \xi , \eta \rangle_A^{\ast}= \langle \eta , \xi \rangle_A\; ,\langle \xi , \eta a \rangle_A = \langle \xi , \eta \rangle_A a$,
for $\xi, \eta \in \Xi$ and $a \in A$.

\begin{definition}\cite{C1} Let $\Xi$, $A$ and ${\mathfrak g}$ be as above.
A \textit{connection} $\nabla$ is a linear map from $\Xi$ to $\Xi \otimes {\mathfrak g}^{\ast}$ such that
\begin{equation*}\label{conn}
\nabla_{X}(\xi a)=(\nabla_{X}(\xi))a +\xi(\delta_{X}(a)),
\end{equation*}
for all $X \in {\mathfrak g}$, $\xi \in \Xi$ and $a\in A$. We say that the connections are \textit{compatible} with the Hermitian metric if
\begin{equation*}\label{comp}
\delta_{X}(\langle \xi, \eta \rangle_A) = \langle \nabla_{X}\xi, \eta \rangle_A + \langle \xi, \nabla_{X}\eta \rangle_A.
\end{equation*}
\end{definition}

We denote the set of compatible connections by $CC(\Xi)$.

According to Connes' theory, we can always define a compatible connection on a projective module over $A$ as follows. For a given unital $C^{\ast}$-algebra $A$ and a projection $Q \in A$, $QA$ is a projective right $A$-module in an obvious way. As described in \cite{C1}, we define a connection on $QA$, called ``Grassmannian connection'', by
\begin{equation*}\label{Grass}
\nabla^{0}_X(\xi)=Q\delta_X(\xi)\;\in QA,\;\;\quad\text{for all}\quad \xi\in QA\quad\text{and}\quad X\in {\mathfrak g}.
\end{equation*}
 Obviously, this is a compatible connection with the canonical Hermitian metric on $QA$, such that $\langle
\xi,\eta\rangle={\xi}^{\ast}\eta$ for $\xi,\eta \in QA$.

For given right $A$-module $\Xi$, let $E=End_A(\Xi)$. Then the following facts are known in \cite{CR}.
If $\nabla$ and ${\nabla}'$ are any two connections, then $\nabla_X - {\nabla}'_X$ is an element of $E$, for each $X \in {\mathfrak g}$. If $\nabla$ and ${\nabla}'$ are both compatible with the Hermitian metric, then $\nabla_X -{\nabla}'_X$ is a skew-symmetric element of $E$ for each $X \in {\mathfrak g}$. Thus, once we have a compatible connection $\nabla$, every other compatible connection ${\nabla}'$ is of the form $\nabla +\mu$, where $\mu$ is a linear map from $\mathfrak{g}$ into $E^s$, a set of skew-symmetric element of $E$, such that  ${\mu_X}^{\ast}=-\mu_X$ for $X\in\mathfrak{g}$.

The curvature of a connection $\nabla$ is defined to be the alternating bilinear form $\Theta_{\nabla}$ on $\mathfrak{g}$, given by
\begin{equation*}\label{cur}
\Theta_{\nabla}(X,Y)=\nabla_{X}\nabla_{Y}-\nabla_{Y}\nabla_{X}-\nabla_{[X,Y]},
\end{equation*}
for $X$,$Y \in \mathfrak{g}$. It is not hard to check that the values of $\Theta$ are in $E$ for a connection $\nabla$, and the values of $\Theta$ are in $E^s$ if a connection $\nabla$ is compatible with respect to the Hermitian metric.

For given $A$-valued inner product $\langle\cdot,\cdot\rangle_A$, we can define an $E$-valued inner product $\langle\cdot,\cdot\rangle_E$ by
\[\langle \xi,\eta\rangle_E\zeta=\xi\langle\eta,\zeta\rangle_A,\]
for $\xi,\eta,\zeta \in \Xi$. So there is a natural bimodule structure (left $E$-right $A$) on $\Xi$.

If A has a faithful $\alpha$-invariant trace, $\tau$, then $\tau$ determines a faithful trace, $\tau_E$, on $E$, defined by
\[\tau_E(\langle \xi,\eta \rangle_E)=\tau(\langle \eta,\xi\rangle_A).\]

To define the Yang-Mills functional on $CC(\Xi)$, we need a bilinear form on the space of alternating 2-forms with values in $E$.  Let $\{Z_1,\dotsb,Z_n\}$ be a basis for $\mathfrak{g}$. We define a bilinear form $\{\cdot,\cdot\}_E$ by
\[\{\Phi,\Psi\}_E=\sum_{\i<j}\Phi(Z_i\wedge Z_j)\Psi(Z_i\wedge Z_j),\]
for alternating $E$-valued 2-forms $\Phi,\Psi$. Clearly, its values are in $E$.

Then the Yang-Mills functional, $YM$, is defined on $CC(\Xi)$ by
\begin{equation}\label{YM}
YM(\nabla)=-\tau_E(\{\Theta_{\nabla},\Theta_{\nabla}\}_E).
\end{equation}

Now let $\Xi$ be the left-$E^{c,\hslash}_{\mu\nu}$ and right-$D^{c,\hslash}_{\mu\nu}$ bimodule described in the previous section.
Let $G$ be the reparametrized Heisenberg Lie group in (\ref{rep-H}) with the Lie algebra $\mathfrak{g}$; then for a positive integer $c$, the basis of ${\mathfrak g}$ is given by $\{X,Y,Z\}$ with $[X,Y]=cZ$, where $X=(0,1,0),Y=(1,0,0),Z=(0,0,1)$.
 Using the Heisenberg group action $G$ on $D^{c,\hbar}_{\mu\nu}$ given by the formula (\ref{H-action3}), we calculate the infinitesimal form, $\delta$, of the action $\mathfrak{g}$ on $D^{c,\hbar}_{\mu\nu}$ in the following way. For $\Phi \in (D^{c,\hbar}_{\mu\nu})^{\infty}$,
\[(\delta_{(r,s,t)}\Phi)(x,y,p)=\frac{d}{dk}L_{(exp(k(r',s',t'))})\Phi(x,y,p)\Bigl|_{k=0}\]
\[=2{\pi}ip(t+cs(x-{\hbar}p\mu))\Phi(x,y,p)-r\frac{\partial\Phi}{\partial x}(x,y,p)-s\frac{\partial\Phi}{\partial y}(x,y,p),\]
where $(r',s',t')\in G$ and $(r,s,t)\in {\mathfrak g}$.
Then the associated derivations $\delta_X, \delta_Y,\delta_Z$ are given by
\begin{equation*}\label{der1}
(\delta_{X}\Phi)(x,y,p)=2\pi icp(x-\hbar p\mu)\Phi(x,y,p)-\frac{\partial\Phi}{\partial y}(x,y,p),
\end{equation*}
\begin{equation*}\label{der2}
(\delta_{Y}\Phi)(x,y,p)=-\frac{\partial\Phi}{\partial x}(x,y,p),
\end{equation*}
\begin{equation*}\label{der3}
(\delta_{Z}\Phi)(x,y,p)=2\pi ipc\,\Phi(x,y,p).
\end{equation*}

To find a compatible connection on the projective module ${\Xi}^{\infty}$, we use the canonical Grassmannian connection ${\nabla}^0$ on $Q(D^{c,\hslash}_{\mu\nu})^{\infty}$ for a projection $Q\in (D^{c,\hslash}_{\mu\nu})^{\infty}$.
As we know, $QD^{c,\hslash}_{\mu\nu}$ and $\Xi$ are projective modules over the same $C^{\ast}$-algebra $D^{c,\hslash}_{\mu\nu}$ on the right. So we can construct a module map between $QD^{c,\hslash}_{\mu\nu}$ and $\Xi$, which preserves the dense subalgebras $Q(D^{c,\hslash}_{\mu\nu})^{\infty}$ and ${\Xi}^{\infty}$ as follows. As shown in Proposition 2.1 in \cite{Rief3}, for given $E^{c,\hslash}_{\mu\nu}$-$D^{c,\hslash}_{\mu\nu}$ bimodule $\Xi$, there exits a function $R \in \Xi$ such that $\langle R,R \rangle_E=Id_E$ and $Q=\langle R,R \rangle_D$ is a projection of $D^{c,\hslash}_{\mu\nu}$ since $E^{c,\hslash}_{\mu\nu}$ and $D^{c,\hslash}_{\mu\nu}$ both have identity elements.  We will show that we can choose $R$ smooth later in this section.
The module isomorphism $\phi : \Xi \longrightarrow QD^{c,\hslash}_{\mu\nu}$ is given by $\phi(f)=\langle R,f \rangle_D$, and the corresponding Grassmannian connection on ${\Xi}^{\infty}$ is given by
\begin{equation}\label{Gconnection}
\nabla^{0}_X(f)=R\cdot \delta_X(\langle R,f \rangle_D ),
\end{equation}
for $ X \in \mathfrak{g}$, $f \in {\Xi}^{\infty}$, and $R \in {\Xi}^{\infty}$ such that $Q=\langle R,R\rangle_D$. It is not hard to check that $\nabla^{0}$ in (\ref{Gconnection}) is compatible with $\langle\cdot,\cdot\rangle_D$.
\begin{remark}
Notice that, since $\langle R,R\rangle_E=Id_E$, the trace of the projection $Q=\langle R,R\rangle_D$ is $2\hslash\mu$.
\end{remark}

For the notational simplicity, we denote ${\Xi}^{\infty}$, $(D^{c,\hslash}_{\mu\nu})^{\infty}$ and $(E^{c,\hslash}_{\mu\nu})^{\infty}$ by $\Xi$, $D^{c,\hslash}_{\mu\nu}$ and $E^{c,\hslash}_{\mu\nu}$ in the rest of this section.

We now compute the corresponding curvature $\Theta^{0}_\nabla(X,Y)$ of
the Grassmannian connection $\nabla^{0}$ for $X,Y \in {\mathfrak
g}$. Since $\Theta^{0}_\nabla(X,Y)$ is
$E^{c,\hslash}_{\mu\nu}$-valued, using formula (\ref{Gconnection})
we can calculate $\Theta^{0}_\nabla(X,Y)\cdot f$ as follows.
\[\Theta^{0}_\nabla(X,Y)\cdot f=(\nabla^0_{X}\nabla^0_{Y}-\nabla^0_{Y}\nabla^0_{X}-\nabla^0_{[X,Y]})\cdot f=\nabla^{0}_{X}(\nabla^{0}_{Y}(f))-\nabla^{0}_{Y}(\nabla^{0}_{X}(f))-\nabla^{0}_{[X,Y]}(f)\]
\[=R\cdot(\delta_X\langle R,R\rangle_D\delta_Y\langle R,f\rangle_D-\delta_Y\langle R,R\rangle_D\delta_X\langle R,f\rangle_D),\]
for $f \in \Xi$, since $R\cdot\langle R,R\rangle_D=R$ and $\delta$ is a Lie algebra homomorphism, i.e.$[\delta_X,\delta_Y]=\delta_{[X,Y]}$.

 At first glance, it seems that the curvature ${\Theta}^0_{\nabla}(X,Y)$ would depend on the value $f$ at which it is evaluated. Since the values of the curvature lie in $(E^{c,\hslash}_{\mu\nu})^s$, the set of skew-symmetric elements of
$E^{c,\hslash}_{\mu\nu}$, we have the following property.
 \begin{lemma}
For the element of $R$ constructed as above,
\[R\cdot(\delta_X\langle R,R\rangle_D\delta_Y\langle R,g\rangle_D-\delta_Y\langle R,R\rangle_D\delta_X\langle
R,g\rangle_D)\]
\[=(R\cdot(\delta_X\langle R,R\rangle_D\delta_Y\langle R,R\rangle_D-\delta_Y\langle R,R\rangle_D\delta_X\langle
R,R\rangle_D))\cdot\langle R,g\rangle_D,\]
$g \in \Xi$ and $X,Y \in \mathfrak{g}$.
\end{lemma}
\begin{proof}
Since the values of $\Theta^{0}_\nabla$ are in $(E^{c,\hslash}_{\mu\nu})^s$, $\Theta^{0}_\nabla(X,Y)$ satisfies
\begin{equation*}
\langle\Theta^{0}_\nabla(X,Y)\cdot f,g\rangle_D+\langle f,\Theta^{0}_\nabla(X,Y)\cdot g\rangle_D=0
\end{equation*}
 for $f,g \in \Xi$ and $X,Y \in \mathfrak{g}$. Then a straightforward calculation shows that
\[\langle f,R\rangle_D\delta_X\langle R,R\rangle_D\delta_Y\langle R,g\rangle_D-\langle f,R\rangle_D\delta_Y\langle
R,R\rangle_D\delta_X\langle R,g\rangle_D\]
\[=\delta_X\langle f,R\rangle_D\delta_Y\langle R,R\rangle_D\langle R,g\rangle_D-\delta_Y\langle f,R\rangle_D\delta_X\langle R,R\rangle_D\langle R,g\rangle_D.\]
Now let $f=R$; then we have
\begin{equation*}
\langle R,R\rangle_D\delta_X\langle R,R\rangle_D\delta_Y\langle R,g\rangle_D-\langle R,R\rangle_D\delta_Y\langle
R,R\rangle_D\delta_X\langle R,g\rangle_D
\end{equation*}
\begin{equation*}
=\delta_X\langle R,R\rangle_D\delta_Y\langle R,R\rangle_D\langle R,g\rangle_D-\delta_Y\langle R,R\rangle_D\delta_X\langle R,R\rangle_D\langle R,g\rangle_D.
\end{equation*}
By applying $R$ on both sides, we obtain the desired equation.
\end{proof}

Therefore, we can develop a formula for  $(E^{c,\hslash}_{\mu\nu})^s$-valued $\Theta^{0}_{\nabla}$ in closed form that does not depend on $f \in \Xi$ as follows.
\begin{proposition}\label{GrassCurv1}
For given basis elements of the Heisenberg Lie algebra ${\mathfrak g}$, $\{X,Y,Z\}$ with $[X,Y]=cZ$, where $c\in \mathbb{Z}^{+}$, we have
\begin{equation*}\label{Grass-curv 1}
\Theta^{0}_{\nabla}(X,Y)=\langle R\cdot(\delta_X\langle R,R\rangle_D\delta_Y\langle R,R\rangle_D-\delta_Y\langle
R,R\rangle_D\delta_X\langle R,R\rangle_D),R\rangle_E,
\end{equation*}
\begin{equation*}\label{Grass-curv 2}
\Theta^{0}_{\nabla}(X,Z)=\langle R\cdot(\delta_X\langle R,R\rangle_D\delta_Z\langle R,R\rangle_D-\delta_Z\langle
R,R\rangle_D\delta_X\langle R,R\rangle_D),R\rangle_E,
\end{equation*}
\begin{equation*}\label{Grass-curv 3}
\Theta^{0}_{\nabla}(Y,Z)=\langle R\cdot(\delta_Y\langle R,R\rangle_D\delta_Z\langle R,R\rangle_D-\delta_Z\langle
R,R\rangle_D\delta_Y\langle R,R\rangle_D),R\rangle_E.
\end{equation*}
\end{proposition}
\begin{proof}
For given basis elements $\{X,Y,Z\}$ with $[X,Y]=cZ$ of the Heisenberg Lie algebra $\mathfrak{g}$, the previous lemma implies that
\begin{equation*}
\Theta^{0}_\nabla(X,Y)\cdot f=\nabla^{0}_{X}(\nabla^{0}_{Y}(f))-\nabla^{0}_{Y}(\nabla^{0}_{X}(f))-\nabla^{0}_{[X,Y]}(f)
\end{equation*}
\begin{equation*}
=R\cdot(\delta_X\langle R,R\rangle_D\delta_Y\langle R,f\rangle_D-\delta_Y\langle R,R\rangle_D\delta_X\langle
R,f\rangle_D
\end{equation*}
\begin{equation*}
=(R\cdot(\delta_X\langle R,R\rangle_D\delta_Y\langle R,R\rangle_D-\delta_Y\langle R,R\rangle_D\delta_X\langle
R,R\rangle_D))\cdot\langle R,f\rangle_D
\end{equation*}
\begin{equation*}
=\langle R\cdot(\delta_X\langle R,R\rangle_D\delta_Y\langle R,R\rangle_D-\delta_Y\langle R,R\rangle_D\delta_X\langle
R,R\rangle_D),R\rangle_E\cdot f.
\end{equation*}
Since this equation holds for every $f \in \Xi$ , we have
\begin{equation*}
\Theta^{0}_{\nabla}(X,Y)=\langle R\cdot(\delta_X\langle R,R\rangle_D\delta_Y\langle R,R\rangle_D-\delta_Y\langle
R,R\rangle_D\delta_X\langle R,R\rangle_D),R\rangle_E.
\end{equation*}
Similarly, we can establish the second and the third equalities.
\end{proof}

The above proposition suggests that the curvature of the Grassmannian connection on $\Xi$ can be computed in terms of a smooth function $R \in \Xi$. In what follows, we will give a specific example of non-trivial smooth function $R\in \Xi$ that produces a Grassmannian curvature $\Theta^{0}_{\nabla}$ that is non-zero. In fact, they can be given by the formulas.
 Since $R$ is chosen to give a projection $Q$ of $D^{c,\hslash}_{\mu\nu}$ such that $Q=\langle
R,R\rangle_D$ and $\langle R,R\rangle_E=Id_E$, using the technique of Rieffel projection we let
\begin{equation*}
\langle
R,R\rangle_D(x,y,p)=g(x,y)\delta_{1}(p)+h(x,y)\delta_{0}(p)+\overline{g}(x+2\hslash\mu,y+2\hslash\nu)\delta_{-1}(p),
\end{equation*}
where $g, h \in \Xi$, and $h$ is real-valued. We also use the formula for $\langle \cdot,\cdot\rangle_D$, and obtain
\begin{enumerate}[({a}-1)]
\item $h(x,y)=\displaystyle{\sum_{k}}|R(x+k,y)|^2$,
\item $g(x,y)=\displaystyle{\sum_{k}}\overline{e}(ck(y-h\nu))R(x+k,y)\overline{R}(x-2\hslash\mu+k,y-2\hslash\nu)$.
\end{enumerate}
Using the fact that $\langle R,R\rangle_D$ is idempotent, we obtain
\begin{enumerate}[({b}-1)]
\item  $g(x,y)g(x-2\hslash\mu,y-2\hslash\nu)=0$,
\item  $g(x,y)[1-h(x,y)-h(x-2\hslash\mu,y-2\hslash\nu)]=0$,
\item  $|g(x,y)|^2+|g(x+2\hslash\mu,y+2\hslash\nu)|^2=h(x,y)-h^2(x,y)$.
\end{enumerate}
We also want $Q$ to be in $D_0$, so we require
\begin{enumerate}[({c}-1)]
\item $h(x,y)=h(x+k,y),\;\;g(x,y)=\overline{e}(ck(y-\hslash\nu))g(x+k,y)$.
\end{enumerate}
Since $\langle R,R\rangle_E=Id_E$, i.e
\begin{equation*}
\;\sum_{k}e(cpk(y-\hslash
k\nu))\overline{R}(x-2k\hslash\mu,y-2k\hslash\nu)R(x-2k\hslash\mu+p,y-2k\hslash\nu)=Id_{\Xi}(x,y)\delta_{0}(p),
\end{equation*}
we have the following.
\begin{enumerate}[({d}-1)]
\item $\displaystyle{\sum_k} |R(x-2k\hslash\mu,y-2k\hslash\nu)|^2=1$ if $p=0$,
\item  $\displaystyle{\sum_{k}e(cpk(y-\hslash
k\nu))\overline{R}(x-2k\hslash\mu,y-2k\hslash\nu)R(x-2k\hslash\mu+p,y-2k\hslash\nu)=0}$ if $p\ne 0$.
\end{enumerate}

%The technique that we will use to find a specific function $R$ is somewhat related to the way to find a scaling function, which is continuous and compactly supported, in wavelet theory, in particular, a Meyer-type wavelet. (See more details in \cite{PR}.) So
Now we assume that $R$ is a compactly supported real-valued function of one variable $x$ for simplicity, and we define $R$ as follows : let $R(x)$ be 0 on $(-2\hslash\mu,-\hslash\mu]$, smooth on $(-\hslash\mu,-\frac{1}{2}\hslash\mu)$ and 1 on
$[-\frac{1}{2}\hslash\mu,0]$, where $|2\hslash\mu|<\frac{1}{2}$, and define $R$ on $[0,2\hslash\mu)$ by
$R(x)=\sqrt{1-|R(x-2\hslash\mu)|^2}$. The bump function argument in differential geometry guarantees the existence of the smooth function $R$.

Then this definition implies that $h(x,y)$ and $g(x,y)$ in (a-1) and (a-2) should have only one term that is not equal to zero, although which term is non-zero depends on $x$. So we have, for $x\in (-\frac{1}{2},\frac{1}{2})$,
\begin{enumerate}[({\textbf{A}}-1)]
\item $h(x)=R^2(x)$, $g(x,y)=R(x)R(x-2\hslash\mu)$.
\end{enumerate}
Also, the condition in (c-1) suggests that we should extend $h$ and $g$ by
\begin{enumerate}[({\textbf{A}}-2)]
\item $h(x)=R^2(x+k)$, $g(x,y)=\overline{e}(k(y-\hslash\nu))R(x+k)R(x+k-2\hslash\mu)$,
\end{enumerate}
for $x\in (-\frac{1}{2}-k,\frac{1}{2}-k)$. With these functions, $ h(x) $ and $g(x,y)$, we can rewrite
equations (b-1)--(b-3) as follows.
\begin{enumerate}[({\textbf{B}}-1)]
\item $R^2(x)R(x-2\hslash\mu)R(x+2\hslash\mu)=0$,
\item $R^2(x)+R^2(x-2\hslash\mu)=1$,
\item $R^2(x-2\hslash\mu)+R^2(x+2\hslash\mu)+R^2(x)-1=0$.
\end{enumerate}
The above conditions ({\textbf{A}-1)--({\textbf{B}-3) imply the following equation.
\begin{enumerate}[({\textbf{C}}-1)]
\item $R(x)R(x-2l\hslash\mu)=0\;\; \text{if}\;\; |l|\geq 2$.
\end{enumerate}
Also, the equations (d-1) and (d-2) give the following.
\begin{enumerate}[({\textbf{C}}-2)]
\item $\displaystyle{\sum_k |R(x-2k\hslash\mu)|^2=1}$.
\end{enumerate}
\begin{enumerate}[({\textbf{C}}-3)]
\item $R(x)R(x+j)=0\;\; \text{if}\;\; j\ne 0$, where $j\in \mathbb{Z}$.
\end{enumerate}

\begin{remark}
The method we used above to construct the function $R$ is often seen as a mollifier technique used within the theory of partial differential equations to create families of smooth dense spanning functions.
\end{remark}

With this special function $R$, we calculate the corresponding Grassmannian curvature as follows.
\begin{lemma}\label{GrassCurv2}For the function $R$ described above,
\begin{equation*}\label{GC-1}
 {\Theta}^{0}_{\nabla}(X,Y)(x,y,p)=f_1(x)\delta_0(p),
\end{equation*}
\begin{equation*}\label{GC-2}
{\Theta}^{0}_{\nabla}(X,Z)(x,y,p)=0,
\end{equation*}
\begin{equation*}\label{GC-3}
{\Theta}^{0}_{\nabla}(Y,Z)(x,y,p)=f_2(x)\delta_0(p),
\end{equation*}
where $f_1$ and $f_2$ are smooth skew-symmetric periodic functions in the sense that $\overline{f_1}(x)=-f_1(x)$ and $\overline{f_2}(x)=-f_2(x)$, and $f_1(x-2\hslash k\mu)=f_1(x)$ and $f_2(x-2\hslash k\mu)=f_2(x)$ for any integer $k$.
\end{lemma}
\begin{proof} By the formula in Proposition \ref{GrassCurv1}, we have, for $X,Y\in\mathfrak{g}$,
\[\Theta^{0}_{\nabla}(X,Y)=\langle R\cdot(\delta_X\langle R,R\rangle_D\delta_Y\langle R,R\rangle_D-\delta_Y\langle
R,R\rangle_D\delta_X\langle R,R\rangle_D),R\rangle_E.\]
The main reason that Grassmannian curvature is only supported at $p=0$ comes from the formula of $\langle\cdot,\cdot\rangle_E$ and the equation (\textbf{C}-3).
 By the related formulas and properties, the Grassmannian curvature with the special function $R$ is given as follows.
\[{\Theta}^{0}_{\nabla}(X,Y)(x,y,p)\]
\[=-\sum_{k}\sum_{q}\sum_{q'}2\pi iq'(x-2\hslash k\mu+2\hslash q\mu-\hslash q'\mu)\Bigl\{R^2(x-2\hslash k\mu)R^2(x-2\hslash k\mu+2\hslash q\mu)\]
\[\times R(x-2\hslash k\mu+2\hslash q\mu-2\hslash q'\mu)R'(x-2\hslash k\mu+2\hslash q\mu-2\hslash q'\mu) + R(x-2\hslash k\mu)R'(x-2\hslash k\mu)\]
\[\times R^2(x-2\hslash k\mu+2\hslash q\mu)R(x-2\hslash k\mu+2\hslash q\mu-2\hslash q'\mu)R'(x-2\hslash k\mu+2\hslash q\mu-2\hslash q'\mu)\Bigl\}\]
\[+\sum_k\sum_q\sum_{q'}e(cq'k(y-\hslash k\nu))2\pi i(q-q')(x-2\hslash k\mu+\hslash(q-q')\mu)\]
\[\times\Bigl\{ R^2(x-2\hslash k\mu)R(x-2\hslash k\mu+2\hslash q\mu)R'(x-2\hslash k\mu+2\hslash q\mu)R^2(x-2\hslash k\mu+2\hslash q\mu-2\hslash q'\mu)\]
\[+ R^2(x-2\hslash k\mu)R^2(x-2\hslash k\mu+2\hslash q\mu)R(x-2\hslash k\mu+2\hslash q\mu-2\hslash q'\mu)R'(x-2\hslash k\mu+2\hslash q\mu-2\hslash q'\mu)\Bigl\}\ne 0.\]
According to the equation (\textbf{C}-1), we know that for fixed integer $k$, $R^2(x-2\hslash k\mu)R^2(x-2\hslash k\mu+2\hslash q\mu)$ is not necessarily zero unless $|q|\ge 2$. So choose $k=0$ and $q'=0$ for simplicity.
Then the first triple term in the above expression for ${\Theta}^{0}_{\nabla}(X,Y)$ becomes zero and the curvature ${\Theta}^{0}_{\nabla}(X,Y)(x,y,p)$ becomes the following expression.
\[\sum_q 4\pi iq(x+\hslash q\mu) R^2(x)R^3(x+2\hslash q\mu)R'(x+2\hslash q\mu).\]

Since $R(x)R(x+2\hslash q\mu)=0$ for $|q|\ge 2$, the previous expression becomes the following.
\[\sum_q 4\pi iq(x+\hslash q\mu) R^2(x)R^3(x+2\hslash q\mu)R'(x+2\hslash q\mu)\]
\[=-4\pi i(x-\hslash\mu)R^2(x)R^3(x-2\hslash\mu)R'(x-2\hslash\mu)+4\pi i(x+\hslash\mu)R^2(x)R^3(x+2\hslash\mu)R'(x+2\hslash\mu)\ne 0.\]

For each $k$, we will have a similar expression as above that is not necessarily zero. Thus the Grassmannian curvature ${\Theta}^{0}_{\nabla}(X,Y)(x,y,p)$ is not trivial.

As seen above, ${\Theta}^{0}_{\nabla}(X,Y)(x,y,p)$ is only supported at $p=0$. Also it is not hard to see that ${\Theta}^{0}_{\nabla}(X,Y)(x,y,p)$ is a one-variable, periodic, complex valued function. So we denote ${\Theta}^{0}_{\nabla}(X,Y)(x,y,p)$ by $f_1(x)\delta_0(p)$ such that $f_1(x-2\hslash l\mu)=f_1(x)$ for an integer $l$. It is clear that $f_1(x)$ is a skew-symmetric function, i.e.  $\overline{f}_1(x)=-f_1(x)$.
Thus we conclude that ${\Theta}^{0}_{\nabla}(X,Y)(x,y,p)=f_1(x)\delta_0(p)$  for a skew-symmetric periodic function $f_1$.
Similarly, we can show that ${\Theta}^{0}_{\nabla}(Y,Z)(x,y,p)=f_2(x)\delta_0(p)$ for a periodic function $f_2$. Also another similar direct calculation shows that ${\Theta}^{0}_{\nabla}(X,Z)(x,y,p)=0$.
\end{proof}

\section{Critical points of the Yang-Mills functional on quantum Heisenberg manifolds}
Let $\Xi$ be the left-$E^{c,\hslash}_{\mu\nu}$ and right-$D^{c,\hslash}_{\mu\nu}$ projective bimodule given in Section 1.

\begin{definition}
For a function $G \in C^{\infty}(\mathbb{T}^2)$, define a \textbf{multiplication-type} element $\mathbf{G}$ of $E^{c,\hslash}_{\mu\nu}$ by
 \[\mathbf{G}(x,y,p)=G(x,y)\delta_0(p).\]
\end{definition}

\begin{remark}
$\mathbf{G}\in\ E^{c,\hslash}_{\mu\nu}$ means $\gamma_k(\mathbf{G})=\mathbf{G}$ for all $k\in \mathbb{Z}$. This implies $G(x-2\hslash k\mu,y-2\hslash k\nu)=G(x,y)$, for $(x,y)\in \mathbb{R}\times \mathbb{T}$. (Here we are identifying $\mathbb{T}^2$ with $\mathbb{R}^2/(2\pi\hslash\mu\mathbb{Z}\times2\pi\hslash\nu\mathbb{Z})$).
Thus $G$ has to be defined on $\mathbb{T}^2$ to produce an element $\mathbf{G}$ of $E^{c,\hslash}_{\mu\nu}$. Also any multiplication-type element $\mathbf{G}$ is a smooth element of $E^{c,\hslash}_{\mu\nu}$ since the corresponding function $G$ is smooth.
\end{remark}
\begin{lemma}\label{skewG}
Let $\mathbf{G}$ be a multiplication-type element of $E^{c,\hslash}_{\mu\nu}$. Then $\mathbf{G}$ is skew-symmetric, i.e $\mathbf{G}^{\ast}=-\mathbf{G}$ if and only if the corresponding function $G\in C^{\infty}(\mathbb{T}^2)$ is also skew-symmetric, i.e. $\overline{G}(x,y)=-G(x,y)$.
\end{lemma}
\begin{proof}The proof is left to the reader.
\end{proof}

\begin{proposition}\label{MultiG}
Let $\mathbf{G}$ be a multiplication-type element of $E^{c,\hslash}_{\mu\nu}$ with a corresponding function $G \in C^{\infty}(\mathbb{T}^2)$.
i.e. $\mathbf{G}(x,y,p)=G(x,y)\delta_0(p)$. Then $\mathbf{G}$ is skew-symmetric if and only if $\mathbf{G}$ acts on ${\Xi}^{\infty}$ as a (skew-symmetric) multiplication operator, i.e. \[(\mathbf{G}\cdot f)(x,y)=-G(x,y)f(x,y)\;\; \text{for}\;\; f \in {\Xi}^{\infty}.\]
\end{proposition}
\begin{proof}
Let $\mathbf{G}$ be a multiplication-type element of $E^{c,\hslash}_{\mu\nu}$, i.e. $\mathbf{G}(x,y,p)=G(x,y)\delta_0(p)$, for $G\in C^{\infty}(\mathbb{T}^2)$. If $\mathbf{G}$ is skew-symmetric, then the corresponding function $G$ is skew symmetric by the previous lemma.
Thus $\overline{G}(x,y)=-G(x,y)$. So $({\mathbf{G}}\cdot f)(x,y)=\sum_q\overline{{\mathbf{G}}}(x,y,q)f(x+q,y)=\sum_q\overline{G}(x,y)\delta_0(q)f(x+q,y)=\overline{G}(x,y)f(x,y)=-G(x,y)f(x,y)$.
Now assume that a multiplication-type element $\mathbf{G}$ of $E^{c,\hslash}_{\mu\nu}$ acts on ${\Xi}^{\infty}$ by $(\mathbf{G}\cdot f)(x,y)$ =$-G(x,y)f(x,y)$ for $f \in {\Xi}^{\infty}$. Then $\sum_q\overline{\mathbf{G}}(x,y,q)f(x+q,y)$=$-G(x,y)f(x,y)$. This implies that $\sum_q\overline{G}(x,y)\delta_0(q)f(x+q,y)=\overline{G}(x,y)f(x,y)=-G(x,y)f(x,y)$. Thus $\overline{G}(x,y)=-G(x,y)$. Therefore, $\mathbf{G}^{\ast}=-\mathbf{G}$ by Lemma \ref{skewG}.
\end{proof}

With this notation and Lemma \ref{GrassCurv2}, we can view the Grassmannian curvature ${\Theta}^{0}_{\nabla}$ as a multiplication-type element of $E^{c,\hslash}_{\mu\nu}$ with the corresponding skew-symmetric functions $f_1$, $0$ and $f_2$. So we write ${\Theta}^{0}_{\nabla}$ as follows.
\begin{equation}\label{GrassSim}
{\Theta}^{0}_{\nabla}(X,Y)={\mathbf{f_1}}\;,\;\;{\Theta}^{0}_{\nabla}(X,Z)=0\;,\;\;{\Theta}^{0}_{\nabla}(Y,Z)={\mathbf{f_2}},
\end{equation}
where  $\mathbf{f_1}(x,y,p)=f_1(x)\delta_0(p)$, $\mathbf{f_2}(x,y,p)=f_2(x)\delta_0(p)$ and $\overline{f}_1(x)=-f_1(x)$, $\overline{f}_2(x)=-f_2(x)$, $f_1, f_2 \in C^{\infty}(\mathbb{T})$.

\begin{proposition}\label{Liebracket}Let $X$, $Y$, $Z$ be the basis of the Heisenberg Lie algebra $\mathfrak{g}$ with $[X,Y]=cZ$.
Let $\nabla^{0}$ be the Grassmannian connection on ${\Xi}^{\infty}$ given in (\ref{Gconnection}) and let ${\mathbf{G}}$ be a multiplication-type skew-symmetric element of $E^{c,\hslash}_{\mu\nu}$ as defined above, corresponding to the smooth skew-symmetric function $G \in C^{\infty}(\mathbb{T}^2)$. Then for $f \in {\Xi}^{\infty}$,
\[([\nabla_X^{0},{\mathbf{G}}]\cdot f)(x,y)=(\frac{\partial}{\partial y}G)(x,y)f(x,y),\]
\[([\nabla_Y^{0},{\mathbf{G}}]\cdot f)(x,y)=(\frac{\partial}{\partial x}G)(x,y)f(x,y),\]
\[([\nabla_Z^{0},{\mathbf{G}}]\cdot f)(x,y)=0.\]
\end{proposition}
\begin{proof}
Let $\mathbf{G}(x,y,p)=G(x,y)\delta_0(p)$ for a skew-symmetric function $G \in C^{\infty}(\mathbb{T}^2)$. Then for $f\in {\Xi}^{\infty}$,
\[[{\nabla}^0_X,\mathbf{G}](f)(x,y)= ({\nabla}^0_X\circ\mathbf{G})(f)(x,y)-(\mathbf{G}\circ{\nabla}^0_X)(f)(x,y)\]
\[={\nabla}^0_X(\mathbf{G}\cdot f)(x,y)-\mathbf{G}\cdot({\nabla}^0_X(f))(x,y)
=R\cdot\delta_X\langle R,\mathbf{G}\cdot f\rangle_D(x,y)-G(x,y)({\nabla}^0_X(f))(x,y)\]
\[=\sum_{q}R(x+2\hslash q\mu)\overline{\delta_X\langle R,\mathbf{G}\cdot f\rangle_D}(x+2\hslash q\mu, y+2\hbar q\nu,q)+G(x,y)(R\cdot\delta_X\langle R,f\rangle_D)(x,y).\]
The important equation that we use in this proof is given in (\textbf{C}-3), $R(x)R(x+j)=0$ if $j\ne 0$. By related formulas and equations, we obtain the following.
\[[{\nabla}^0_X,\mathbf{G}](f)(x,y)=\sum_{q}R^2(x+2\hslash q\mu)(\frac{\partial}{\partial y}G(x,y))f(x,y).\]
 By equation (\textbf{C}-2), it follows that $([\nabla_X^{0},{\mathbf{G}}]\cdot f)(x,y)=(\frac{\partial}{\partial y}G)(x,y)f(x,y)$.
Similarly, we can obtain the second equation and the third equation. We leave the details to the reader.
\end{proof}

The previous two propositions imply the following.
\begin{corollary}\label{GCR}  Let $X$,$Y$,$Z$ be the basis of the Heisenberg Lie algebra $\mathfrak{g}$ with $[X,Y]=cZ$. Let ${\nabla}^0$ be the Grassmannian connection given in (\ref{Gconnection}) and $\mathbf{G}$ be a multiplication-type skew-symmetry element of $E^{c,\hslash}_{\mu\nu}$. Then
\[[\nabla_X^{0},{\mathbf{G}}](x,y,p)=-\frac{\partial}{\partial y}\mathbf{G}(x,y,p),\;[\nabla_Y^{0},{\mathbf{G}}](x,y,p)=-\frac{\partial}{\partial x}\mathbf{G}(x,y,p),\;[\nabla_Z^{0},{\mathbf{G}}]=0.\]
\end{corollary}
\begin{proof}
It is obvious by Proposition \ref{Liebracket} and Proposition \ref{MultiG}.
\end{proof}

For notational simplicity we write the equations in Corollary \ref{GCR} as follows.
\[[\nabla_X^{0},{\mathbf{G}}]=-\frac{\partial}{\partial y}\mathbf{G}\;\;,\;\; [\nabla_Y^{0},{\mathbf{G}}]=-\frac{\partial}{\partial x}\mathbf{G}\;\;,\;\; [\nabla_Z^{0},{\mathbf{G}}]=0.\]

\begin{proposition}\label{Lie algebra by GC}
The Grassmannian connection $\nabla^{0}$ given in (\ref{Gconnection}) generates an infinite dimensional Lie algebra, in the sense that the Lie algebra of operators generated by ${\nabla}^0_X$,${\nabla}^0_Y$,${\nabla}^0_Z$ is infinite dimensional.
\end{proposition}
\begin{proof}
Using the formula for the curvature and the notation for the Grassmannian curvature in (\ref{GrassSim}), we have the following.
\[[\nabla_X^{0},\nabla_Y^{0}]={\Theta}_{\nabla}^{0}(X,Y)+\nabla_Z^{0}={\mathbf{f_1}}+\nabla_Z^{0},\]
\[[\nabla_X^{0},\nabla_Z^{0}]={\Theta}_{\nabla}^{0}(X,Z)=0 \quad \text{and}\quad
[\nabla_Y^{0},\nabla_Z^{0}]={\Theta}_{\nabla}^{0}(Y,Z)={\mathbf{f_2}}.\]
Also, by the previous corollary we have
\[[\nabla_X^{0},{\mathbf{f_1}}]=0,\;\;[\nabla_Y^{0},{\mathbf{f_1}}]=-\frac{\partial}{\partial x}\mathbf{f_1},\;\;[\nabla_Z^{0},{\mathbf{f_1}}]=0,\]
\[[\nabla_X^{0},{\mathbf{f_2}}]=0,\;\;[\nabla_Y^{0},{\mathbf{f_2}}]=-\frac{\partial}{\partial x}\mathbf{f_2},\;\;[\nabla_Z^{0},{\mathbf{f_2}}]=0.\]
The proof of Lemma \ref{GrassCurv2} shows that $-\frac{\partial}{\partial x}\mathbf{f_1}\ne 0$ and $-\frac{\partial}{\partial x}\mathbf{f_2}\ne 0$, in general.
Thus $\nabla^0_X, \nabla^0_Y, \nabla^0_Z $ generate an infinite dimensional Lie algebra.
\end{proof}

\section{The Yang-Mills functional on quantum Heisenberg manifolds}

The Yang-Mills problem is mainly about determining the nature of the set of the critical points for the Yang-Mills functional $YM$. In particular, the critical points where $YM$ attains its minimum. According to differential calculus, $\nabla$ is a critical point of $YM$ if $D(YM(\nabla))=0$, i.e the derivative of $YM$ at $\nabla$ is zero. Also we have
\begin{equation*}
\frac{d}{dt}\Bigl\vert_{t=0}YM(\nabla+t\mu)=D(YM(\nabla))\cdot\mu,
\end{equation*}
where $D$ is the derivative of $YM$. So $\nabla$ is a critical point of $YM$ if we have, for all linear maps $\mu : \mathfrak{g} \to E^s$,
\[\frac{d}{dt}\Bigl\vert_{t=0}YM(\nabla+t\mu)=0.\]

Thus, as given in \cite{Rief2}, for given Lie algebra $\mathfrak{g}$, $\nabla$ is a critical point of $YM$ if for all $Z_i\in\mathfrak{g}$,
\begin{equation}\label{CPC}
\sum_{j}[\nabla_{Z_i},\Theta_{\nabla}(Z_i\wedge Z_j)]-\sum_{j<k}c^i_{jk}\Theta_{\nabla}(Z_j\wedge Z_k)=0,
\end{equation}
where $c^{i}_{jk}$ are structure constants of $\mathfrak{g}$.

Recall that our Lie algebra ${\mathfrak g}$ is the Heisenberg Lie algebra with three basis
elements $\{X,Y,Z\}$ satisfying $[X,Y]=cZ$ for a positive integer $c$. This together with (\ref{CPC}) gives the following.
The connection $\nabla$ will be a critical point if
\begin{equation}\label{EQ1}
[\nabla_Y,\Theta_{\nabla}(X,Y)]+[\nabla_Z,\Theta_{\nabla}(X,Z)]=0,
\end{equation}
\begin{equation}\label{EQ2}
[\nabla_X,\Theta_{\nabla}(Y,X)]+[\nabla_Z,\Theta_{\nabla}(Y,Z)]=0,
\end{equation}
\begin{equation}\label{EQ3}
[\nabla_X,\Theta_{\nabla}(Z,X)]+[\nabla_Y,\Theta_{\nabla}(Z,Y)]-c\cdot\Theta_{\nabla}(X,Y)=0.
\end{equation}
It is now easy to see that the Grassmannian connection ${\nabla}^{0}$ given in (\ref{Gconnection}) is
not a critical point of $YM$. But we know that any other compatible connection can be obtained from the Grassmannian connection ${\nabla}^0$ by adding a linear map $\mu$ from $\mathfrak{g}$ into $(E^{c,\hslash}_{\mu\nu})^s$, a set of skew-symmetric element of $E^{c,\hslash}_{\mu\nu}$.
It is not hard to find a concrete example of a smooth skew-symmetric element of $E^{c,\hslash}_{\mu\nu}$, but it is difficult to compute the left hand sides of (\ref{EQ1}), (\ref{EQ2}) and (\ref{EQ3}) in general.  So we will try the simplest form of such a linear map whose range is skew-symmetric. In particular, we will use a linear map $\mu : \mathfrak{g}\rightarrow (E^{c,\hslash}_{\mu\nu})^s$ whose range lies in the set of multiplication-type element of $E^{c,\hslash}_{\mu\nu}$ introduced in the previous section.

For given Grassmannian connection ${\nabla}^0$ with the curvature ${\Theta}^0_{\nabla}$ as before, let $\nabla={\nabla}^{0}+\mu$, where ${\mu}_{X}^{\ast}=-{\mu}_X \in (E^{c,\hslash}_{\mu\nu})^{\infty}$ for $X \in {\mathfrak g}$. Then the corresponding curvature ${\Theta}_{\nabla}$ of $\nabla$ is the following. For $X,Y,Z \in {\mathfrak g}$ with $[X,Y]=cZ$,
\begin{equation}\label{curv1}
{\Theta}_{\nabla}(X,Y)={\Theta}_{\nabla}^{0}(X,Y)+[\nabla_X^{0},\mu_Y]-[\nabla_Y^{0},\mu_X]+[\mu_X,\mu_Y]-\mu_{[X,Y]},
\end{equation}
\begin{equation}\label{curv2}
{\Theta}_{\nabla}(X,Z)={\Theta}_{\nabla}^{0}(X,Z)+[\nabla_X^{0},\mu_Z]-[\nabla_Z^{0},\mu_X]+[\mu_X,\mu_Z],
\end{equation}
\begin{equation}\label{curv3}
{\Theta}_{\nabla}(Y,Z)={\Theta}_{\nabla}^{0}(Y,Z)+[\nabla_Y^{0},\mu_Z]-[\nabla_Z^{0},\mu_Y]+[\mu_Y,\mu_Z].
\end{equation}
Now consider the case where $\mu_X$ is a multiplication-type, skew-symmetric element. Let $\mu_X={\mathbf G_X}$, $\mu_Y={\mathbf G_Y}$ and $\mu_Z={\mathbf G_Z}$ for ${\mathbf G_X}\in E^{c,\hslash}_{\mu\nu}$ for $X\in\mathfrak{g}$ with the corresponding $G_i \in C^{\infty}(\mathbb{T}^2)$ such that $\overline{G}_i(x,y)=-G_i(x,y)$ for $i=1,2,3$.
 Using the formulas in the proof of Proposition \ref{Lie algebra by GC}, we can write the curvature $\Theta_{\nabla}$ in terms of the Grassmannian curvature ${\Theta}^0_{\nabla}$ and the multiplication-type elements $\mathbf{G}_i$. Then the proposed conditions for obtaining critical points, (\ref{EQ1}), (\ref{EQ2}) and (\ref{EQ3}) give the following two equations.
\begin{equation}\label{G3-1}
c\cdot G_3(x,y)=\displaystyle{-\frac{\partial}{\partial y}G_2(x,y)+\frac{\partial}{\partial x}G_1(x,y)+f_1(x)+c_1\cdot i},\;\; \text{for some real number $c_1$}.
\end{equation}
\begin{equation}\label{Elliptic}
\displaystyle{\frac{\partial^2}{\partial y^2}G_3(x,y)+\frac{\partial^2}{\partial x^2}G_3(x,y)=\frac{\partial}{\partial x}f_2(x)+c\cdot c_1\cdot i}.
\end{equation}

Let $w(x)=\frac{\partial}{\partial x}f_2(x)+c\cdot c_1\cdot i$. Then it is obvious that $\overline{w}(x)=-w(x)$ since $f_2$ is skew-symmetric.
To solve the elliptic equation (\ref{Elliptic}), take the Fourier transform on both sides in (\ref{Elliptic}). Then
\begin{equation*}
(-m^2-n^2)\widehat{G}_3(n,m)=\widehat{w}(n)\quad \text{so}\quad
\widehat{G_3}(n,m)=-\frac{\widehat{w}(n)}{(m^2+n^2)}\;,
\end{equation*} for $(m,n)\ne (0,0)$. Therefore
\begin{equation*}\label{Solution G-3}
G_3(x,y)=\sum_{n,m \in {\mathbb
Z},(n,m)\ne (0,0)}\widehat{G_3}(n,m) e(nx)e(my),
\end{equation*}
where $e(x)=exp\,(2\pi i\mu\hslash x)$ and $e(y)=exp\,(2\pi i\nu\hslash y)$. Notice that $G_3$ is periodic and $\widehat{\overline{w}}(n)=\overline{\widehat{w}}(-n)$, so $G_3$ is a skew-symmetric element of $C^{\infty}(\mathbb{T}^2)$, i.e. $\overline{G}_3(x,y)=-G_3(x,y)$. With the solution $G_3$ as above, we can show that the equation (\ref{G3-1}) allows to choose non-trivial $G_1$ and $G_2$.

 We finally state the main theorem.
 \begin{theorem}
Let $\mathfrak{g}$ be the Heisenberg Lie algebra with the basis $X,Y,Z$ satisfying $[X,Y]=cZ$ for a positive integer $c$. Let $\Xi$ be the $E^{c,\hslash}_{\mu\nu}$-$D^{c,\hslash}_{\mu\nu}$ bimodule given as before.
Let ${\nabla}^0$ be the Grassmannian connection on ${\Xi}^{\infty}$ produced by a special function $R$ with the Grassmannian curvature ${\Theta}^0_{\nabla}$ such that
${\Theta}^0_{\nabla}(X,Y)=\mathbf{f_1}$, ${\Theta}^0_{\nabla}(X,Z)=0$ and ${\Theta}^0_{\nabla}(Y,Z)=\mathbf{f_2}$, where $\mathbf{f_1}(x,y,p)=f_1(x)\delta_0(p)$, $\mathbf{f_2}(x,y,p)=f_2(x)\delta_0(p)$ for smooth periodic functions $f_1$ and $f_2$.
Let $\mathbf{G}$ be a linear map on $\mathfrak{g}$ whose range lies in the set of multiplication-type, skew-symmetric elements of $End_{D^{c,\hslash}_{\mu\nu}}(\Xi)=E^{c,\hslash}_{\mu\nu}$. Then $\nabla={\nabla}^0 + {\mathbf G}$ is a critical point of the Yang-Mills functional if and only if the corresponding skew-symmetric periodic functions $G_1$, $G_2$ and $G_3$ satisfy the following equations.
\begin{equation}\label{main1}
 \frac{\partial}{\partial x}G_1(x,y)-\frac{\partial}{\partial y}G_2(x,y)=c\cdot G_3(x,y)-\widetilde{f_1}(x),
 \end{equation}
\begin{equation}\label{main2}
\frac{\partial^2}{\partial y^2}G_3(x,y)+\frac{\partial^2}{\partial x^2}G_3(x,y)=\frac{\partial}{\partial x}f_2(x)+c\cdot a_0,
\end{equation}
where $\mathbf{G_X}(x,y,p)=G_1(x,y)\delta_0(p)$, $\mathbf{G_Y}(x,y,p)=G_2(x,y)\delta_0(p)$, $\mathbf{G_Z}(x,y,p)=G_3(x,y)\delta_0(p)$, and  $\widetilde{f_1}(x)=f_1(x)-a_0$ and $a_0=\int_{\mathbb{T}}f_1(x)\;dx$.
\end{theorem}
\begin{proof}
The proof follows from the previous argument.
\end{proof}
\begin{remark}
We note that these critical points are not minima, but inflection points.
\end{remark}

\section{Laplace's equation on quantum Heisenberg manifolds}
 In \cite{Ros}, J. Rosenberg described the Euler-Lagrangian equation for critical points of the energy functional on the non-commutative torus as Laplace's equation $\Delta a=0$ for a self adjoint element $a$ of the non-commutative torus. In this section, we show that those two equations (\ref{main1}) and (\ref{main2}) in the main theorem of this paper can be expressed in terms of this Laplacian and derivations on quantum Heisenberg manifolds, by using a result of Morita equivalence for quantum Heisenberg manifolds.
First we give the explicit formula of the Laplacian first defined by N. Weaver in \cite{W} as follows.
 For $\Phi \in (D^{c,\hslash}_{\mu\nu})^{\infty}$,
 \[\Delta(\Phi)(x,y,p)=(\delta_X)^2(\Phi)(x,y,p)+(\delta_Y)^2(\Phi)(x,y,p) =-4\pi ip^2(x-\hslash p\mu)^2\Phi(x,y,p)\]
 \[-2\pi ip(x-\hslash p\mu)\frac{\partial\Phi}{\partial y}(x,y,p)-2\pi ip(x-\hslash p\mu)\frac{\partial\Phi}{\partial y}(x,y,p)+\frac{\partial^2\Phi}{\partial y^2}(x,y,p)+\frac{\partial^2\Phi}{\partial x^2}(x,y,p).\]

 When $p=0$, then the above Laplacian becomes the following :
 \[\Delta(\Phi)(x,y,0)=\frac{\partial^2\Phi}{\partial y^2}(x,y,0)+\frac{\partial^2\Phi}{\partial x^2}(x,y,0),\]
for $\Phi \in (D^{c,\hslash}_{\mu\nu})^{\infty}$.
Thus, if we restrict the Laplacian on a smooth element $\mathbf{H}$ of $D^{c,\hslash}_{\mu\nu}$ that is only supported at $p=0$, i.e. $\mathbf{H}(x,y,p)=H(x,y)\delta_0(p)$, $H\in C^{\infty}(\mathbb{T}^2)$, then we have
\[\Delta(\mathbf{H})(x,y,p)=\frac{\partial^2\mathbf{H}}{\partial y^2}(x,y,p)+\frac{\partial^2\mathbf{H}}{\partial x^2}(x,y,p).\]

 As shown in Section 2, the generalized fixed point $C^{\ast}$-algebra $E^{c,\hslash}_{\mu\nu}$ is isomorphic to a quantum Heisenberg manifold with different parameters, in particular $D^{c,\hslash}_{\frac{1}{4\mu}\frac{\nu}{2\mu}}$, so we can treat an $E^{c,\hslash}_{\mu\nu}$-valued multiplication operator $\mathbf{G}$ as a $D^{c,\hslash}_{\frac{1}{4\mu}\frac{\nu}{2\mu}}$-valued multiplication operator in the main theorem.
Also, the Grassmannian curvature ${\Theta}^0_{\nabla}$ that is an $(E^{c,\hslash}_{\mu\nu})^{\infty}$-valued 2-form now can be considered as a $(D^{c,\hslash}_{\frac{1}{4\mu}\frac{\nu}{2\mu}})^{\infty}$-valued 2-form. Thus we obtain the following corollary.

 \begin{corollary} Let $\mathfrak{g}$ and $\Xi$ be as before. Let $\delta$ be the infinitesimal form of the Heisenberg group action given as before and $\Delta={\delta}^2_X+{\delta}^2_Y$ for $X,Y\in \mathfrak{g}$.
 Let ${\nabla}^0$ be the Grassmannian connection on ${\Xi}^{\infty}$ produced by the special function $R$ given in (\ref{Gconnection}). Let ${\Theta}^0_{\nabla}$ be the corresponding Grassmannian curvature such that ${\Theta}^0_{\nabla}(X,Y)$ and ${\Theta}^0_{\nabla}(Y,Z)$ are non-trivial multiplication-type elements of $E^{c,\hslash}_{\mu\nu}$, and ${\Theta}^0_{\nabla}(X,Z)=0$.  Let $\mathbf{G}$ be a linear map on $\mathfrak{g}$ whose range lies in the set of multiplication-type skew-symmetric elements of $End_{D^{c,\hslash}_{\mu\nu}}(\Xi)=E^{c,\hslash}_{\mu\nu}$.
 Then, $\nabla={\nabla}^0 + {\mathbf G}$ is a critical point of Yang-Mills functional if and only if $\mathbf{G_X}$, $\mathbf{G_Y}$ and $\mathbf{G_Z}$ satisfy the following equations.
 \begin{equation*}
 \delta_X(\mathbf{G_Y})-\delta_Y(\mathbf{G_X})=c\cdot\mathbf{G_Z}-{\Theta}^0_{\nabla}(X,Y)+a_0,
 \end{equation*}
 \begin{equation*}
 \Delta(\mathbf{G_Z})=-\delta_Y({\Theta}^0_{\nabla}(Y,Z))+c\cdot a_0,
 \end{equation*}
 where $a_0=\int_{\mathbf{T}}f_1(x)dx$, and ${\Theta}^0_{\nabla}(X,Y)(x,y,p)=f_1(x)\delta_0(p)$ for a smooth periodic function $f_1(x)$.
 \end{corollary}

\end{document}